\documentclass[11pt]{article}
\usepackage{amsmath,amsthm,amssymb,bbm}
\usepackage{fullpage}
\usepackage{authblk} 
\usepackage{color}
\title{\textbf{A Weighted Spectral Quantum Fidelity}}

\author[1]{Cong Trinh Le\thanks{lecongtrinh@qnu.edu.vn}}
\author[2]{The Khoi Vu\thanks{vtkhoi@math.ac.vn}}
\author[3]{Minh Toan Ho\thanks{hmtoan@math.ac.vn}}
\author[4]{Trung Hoa Dinh\thanks{Corresponding author. Email: thdinh@troy.edu}}
\affil[1]{Department of Mathematics and Statistics, Quy Nhon University, Quy Nhon, Vietnam}
\affil[2]{Institute of Mathematics, VAST, 18 Hoang Quoc Viet, Hanoi,  Vietnam}

\affil[3]{Institute of Mathematics, VAST, 18 Hoang Quoc Viet, Hanoi,  Vietnam}
\affil[4]{Department of Mathematics and Statistics, Troy University, Troy, USA}

\date{} 

\newtheorem{theorem}{Theorem}[section]
\newtheorem{definition}{Definition}[section]

\newtheorem{proposition}[theorem]{Proposition}
\newtheorem{corollary}[theorem]{Corollary}
\newtheorem{remark}[theorem]{Remark}

\newcommand{\Tr}{\operatorname{Tr}}
\newcommand{\sharpop}{\mathbin{\sharp}}
\newcommand{\naturalop}{\mathbin{\natural}}

\newcommand{\supp}{\operatorname{supp}}

\newcommand{\keywords}[1]{\par\smallskip\noindent\textbf{Keywords: }#1\par}

\begin{document}
\maketitle

\begin{abstract}
 
We introduce and study a one-parameter family of fidelity-type quantities based on the weighted spectral geometric mean, which we call the \emph{weighted spectral fidelity}
\(
\mathsf{F}_t^{\mathrm{spec}}(\rho,\sigma):=\Tr\!\big[\rho(\rho^{-1}\sharp\sigma)^{2t}\big],\ t\in[0,1].
\)
This family interpolates smoothly between the trivial overlap ($t=0,1$) and the Uhlmann (root) fidelity at $t=\tfrac12$, and it is distinct from the sandwiched Rényi family except at this midpoint. We establish core structural features-unitary invariance, tensor stabilization and multiplicativity, flip symmetry, endpoint behavior, and a orthogonality criterion. We further show explicit \emph{violations of DPI} for generic $t\neq\tfrac12$. For concavity in the state variables we obtain concavity in each variable separately. Closed forms are obtained for pure states and for qubits in Bloch coordinates. We also extend the first Fuchs--van de Graaf inequality to $\mathsf{F}_t^{\mathrm{spec}}$ for all $t\in[0,1]$, while the second inequality fails away from the midpoint.  

\end{abstract}


\keywords{Uhlmann fidelity, Matsumoto fidelity, spectral geometric mean, concavity, log-convexity,  qubits}

\section{Introduction}

Quantifying the similarity between quantum states is a fundamental task in quantum information theory, with applications ranging from quantum hypothesis testing and error correction to quantum cryptography and quantum algorithms. A variety of fidelity-type measures have been introduced to formalize this notion of similarity. Among them, the \emph{Uhlmann fidelity}~\cite{Uhlmann76}, also called the transition probability, has emerged as one of the most widely used. It admits multiple equivalent formulations, such as
{\color{blue} \[
\mathsf{F}_{\mathrm{U}}(\rho,\sigma)= \Tr\sqrt{\rho^{1/2}\sigma\rho^{1/2}},
\] }
and satisfies a desirable list of axioms including symmetry, unitary invariance, monotonicity under quantum channels, and joint concavity~\cite{Jozsa94}. These properties make it an indispensable tool in modern quantum theory.

Beyond Uhlmann’s fidelity, other measures have been studied. Holevo introduced an alternative quantity in the 1970s, sometimes called the \emph{just-as-good fidelity} or \emph{Holevo fidelity}~\cite{Holevo72}, which was later rediscovered under different names including the pretty-good fidelity~\cite{BarnumKnill02}, the $A$-fidelity~\cite{ChenAlbeverioFei02}, and the affinity~\cite{NielsenChuang00}. A comprehensive overview of such fidelity measures is provided by Liang et al.~\cite{LiangEtAl19}. Each fidelity encapsulates a different operational or geometric perspective on quantum state similarity. For instance, the Uhlmann fidelity is linked to quantum-to-classical measurement scenarios, while Holevo’s fidelity connects to certain recovery tasks~\cite{Wilde18}.

In 2010, Matsumoto introduced yet another fidelity function, which can be expressed compactly in terms of the matrix geometric mean:
\[
\mathsf{F}_{\mathrm{M}}(\rho,\sigma)=\Tr(\rho\sharp\sigma),
\qquad
\rho\sharp\sigma:=\rho^{1/2}\!\left(\rho^{-1/2}\sigma\rho^{-1/2}\right)^{1/2}\!\rho^{1/2},
\]
for invertible density operators~\cite{Matsumoto10,Matsumoto11,Matsumoto13}. This measure, now known as the \emph{Matsumoto fidelity}, satisfies many of the same axioms as Uhlmann’s fidelity and Holevo’s fidelity, but offers a dual interpretation: it can be viewed as the maximum classical fidelity associated with classical-to-quantum preparation procedures~\cite{CreeSikora20}. Despite its elegance, the Matsumoto fidelity remained relatively underexplored until the more recent analysis of Cree and Sikora~\cite{CreeSikora20}, who studied it through the lens of semidefinite programming and provided a geometric interpretation in terms of the Riemannian metric on the cone of positive definite matrices.

These developments highlight the central role of the matrix geometric mean~\cite{AndoHiai10,Bhatia07} in bridging linear algebra, operator theory, and quantum information. It appears not only in the definition of Matsumoto fidelity, but also in alternative characterizations of Uhlmann fidelity (e.g.\ Alberti’s variational principle~\cite{Alberti83}) and in inequalities fundamental to quantum information~\cite{Ando79}. Building on this framework, we introduce and analyze a \emph{spectral fidelity family}:
\[
\mathsf{F}_t^{\mathrm{spec}}(\rho,\sigma)=\Tr\!\big(\rho(\rho^{-1}\sharp\sigma)^{2t}\big),
\qquad t\in[0,1],
\]
which interpolates between trivial overlaps and Uhlmann fidelity at $t=\tfrac12$. This family inherits many desirable structural properties such as unitary invariance, tensor product multiplicativity, concavity, log-convexity in $t$, and a natural flip symmetry
\(
\mathsf{F}_t^{\mathrm{spec}}(\rho,\sigma)=\mathsf{F}_{1-t}^{\mathrm{spec}}(\sigma,\rho).
\)

Our aim in this work is to place $\mathsf{F}_t^{\mathrm{spec}}$ in the broader landscape of quantum fidelity measures, to establish its structural properties rigorously, and to highlight connections to Ulhmann and Matsumoto fidelities.  In particular, we show how the parameter $t$ interpolates between different operational regimes.

\medskip

 In Section~\ref{sec:prelim} we recall basic definitions and properties of matrix and spectral geometric means.
Section~\ref{sec:properties} introduces the weighted spectral fidelity $\mathsf{F}_t^{\mathrm{spec}}$ and establishes its fundamental structural properties, including flip symmetry, unitary and tensor invariances, multiplicativity, endpoint behavior, and the orthogonality (zero-support) criterion. Section~\ref{sec:qubit-comparison} derives closed forms of $\mathsf{F}_t^{\mathrm{spec}}$ when one state is pure and presents Bloch-sphere expressions for qubit systems.
Section~\ref{sec:concavity-dpi} analyzes concavity and log-convexity using operator perspective techniques and Hadamard's three-lines theorem, proving \emph{separate concavity in each state variable} for all $t\in[0,1]$ and \emph{log-convexity in the parameter $t$}.  
We provide explicit counterexamples showing that the \emph{data processing inequality fails} for all $t\neq \tfrac{1}{2}$. Section~\ref{sec:fvg} extends the first Fuchs--van de Graaf inequality to all $t\in[0,1]$ and shows the second fails away from the midpoint.

\section{Preliminaries}\label{sec:prelim}

Let $\mathbb{M}_n$ denote the algebra of $n\times n$ complex matrices, and let $\mathbb{P}_n$ denote the cone of positive definite matrices in $\mathbb{M}_n$. For $A,B\in \mathbb{P}_n$, the \emph{matrix geometric mean} was introduced by Pusz and Woronowicz~\cite{PuszWoronowicz75} in 1975 as  
\[
A \sharp B \;=\; A^{1/2}\bigl(A^{-1/2}BA^{-1/2}\bigr)^{1/2}A^{1/2}.
\]
This mean is uniquely characterized as the positive solution $X$ of the Riccati equation
\[
XA^{-1}X = B.
\]

\smallskip

In 1997, Fiedler and Pták~\cite{FiedlerPtak97} proposed the \emph{spectral geometric mean}, defined by  
\[
A \natural B \;=\; (A^{-1}\sharp B)^{1/2}\,A\,(A^{-1}\sharp B)^{1/2}.
\]
{\color{blue} The terminology ``spectral mean'' originates from a fundamental eigenvalue property of this
construction. A key feature is that $(A \natural B)^2$ is similar to the product $AB$. Indeed, using
the identity
\[
A^{-1}\sharp B
= A^{-1/2}\,(A^{1/2}BA^{1/2})^{1/2}\,A^{-1/2},
\]
one checks directly that
\[
(A\natural B)^2
\sim A^{1/2}(A^{-1/2}BA^{-1/2})A^{1/2} = AB,
\]
where $\sim$ denotes similarity of matrices. Since similarity preserves eigenvalues, it follows that
\[
\lambda\bigl((A\natural B)^2\bigr)=\lambda(AB).
\]
As $A\natural B$ is positive definite, its eigenvalues are therefore the positive square roots of the
eigenvalues of $AB$, namely
\[
\lambda(A\natural B)=\sqrt{\lambda(AB)}.
\]
This eigenvalue correspondence is the primary motivation for the term ``spectral mean.''

}  Moreover, the factor $A^{-1}\sharp B$ has a variational characterization: it is the unique minimizer of  
\[
F(X)=\operatorname{Tr}(AX)+\operatorname{Tr}(BX^{-1}), \qquad X>0,
\]
see, e.g.,~\cite{Bhatia07}.

\smallskip  
In 2015, Kim and Lee~\cite{KimLee2015} introduced the \emph{weighted spectral mean}, given by  
\[
A \natural_t B \;=\; (A^{-1}\sharp B)^t \, A \, (A^{-1}\sharp B)^t, \qquad t\in[0,1],
\]
which provides a continuous interpolation between $A$ and $B$ through spectral means. This family has subsequently been linked to the study of relative operator entropy and to several operator inequalities~\cite{tam2,KimLee2015,Kim2021}.

{\color{blue} 
\begin{definition}[Weighted spectral fidelity]
For density operators $\rho,\sigma$, we define the \emph{weighted spectral fidelity} by
\[
\mathsf{F}_t^{\mathrm{spec}}(\rho,\sigma)
:=\Tr(\rho\natural_t\sigma)
=\Tr\!\big(\rho\,(\rho^{-1}\sharp\sigma)^{2t}\big),
\qquad t\in[0,1].
\]
\end{definition}

\begin{remark}
Throughout the paper, the superscript ``$\mathrm{spec}$'' emphasizes that the quantity
$\mathsf{F}_t^{\mathrm{spec}}(\rho,\sigma)$ is a fidelity functional derived from the spectral
geometric mean, and should not be confused with the operator mean $\natural_t$ itself.
\end{remark}
}

The quantity $\mathsf{F}_t^{\mathrm{spec}}(\rho,\sigma)$ should not be confused with the \emph{sandwiched Rényi divergence}~\cite{MuellerReeb16, Beigi13}. For $\alpha>0$, $\alpha\neq 1$, the sandwiched Rényi relative entropy is defined by  
\[
\widetilde{D}_{\alpha}(\rho\|\sigma)
= \frac{1}{\alpha-1}\log\,
\Tr\!\left[
\Big( \sigma^{\tfrac{1-\alpha}{2\alpha}} \rho \,
\sigma^{\tfrac{1-\alpha}{2\alpha}} \Big)^{\alpha}
\right].
\]
In the special case $\alpha=\tfrac12$, one has
\[
\widetilde{D}_{1/2}(\rho\|\sigma) \;=\; -2\log \Tr(\rho^{1/2}\sigma^{1/2}),
\]
so that the sandwiched Rényi divergence reduces to minus twice the logarithm of the Uhlmann fidelity. By contrast, our weighted fidelity is defined as  
\[
\mathsf{F}_t^{\mathrm{spec}}(\rho,\sigma) \;=\; 
\Tr\!\left[\rho\,(\rho^{-1}\sharp\sigma)^{2t}\right],
\]
which at $t=\tfrac12$ coincides with the Uhlmann fidelity. For general values of $t$, however, the two constructions differ: sandwiched Rényi divergences interpolate via conjugation by powers of $\sigma$, while the weighted spectral fidelity interpolates through Kubo--Ando operator means. Thus the two families intersect at the midpoint but otherwise yield distinct one-parameter generalizations of quantum fidelity.

\section{Properties of the weighted spectral fidelity}\label{sec:properties}
This section records properties of the weighted spectral fidelity that mirror those of the Uhlmann fidelity. Throughout, $\rho$ and $\sigma$ denote density operators (positive semidefinite with unit trace), representing quantum states.

\begin{proposition}[Endpoint values and the Uhlmann point]\label{prop:endpoints}
For all density matrices $\rho,\sigma$,
\[
\mathsf{F}_0^{\mathrm{spec}}(\rho,\sigma)=\mathsf{F}_1^{\mathrm{spec}}(\rho,\sigma)=1,
\qquad
\mathsf{F}_{1/2}^{\mathrm{spec}}(\rho,\sigma)=\Tr\sqrt{\rho^{1/2}\sigma\rho^{1/2}}=F(\rho,\sigma).
\]
\end{proposition}

\begin{proof}
$\mathsf{F}_0^{\mathrm{spec}}=\Tr\rho=1$. For $t=1$, write with $M:=\rho^{1/2}\sigma\rho^{1/2}$,
\[
\rho^{-1}\sharp\sigma=\rho^{-1/2}M^{1/2}\rho^{-1/2},
\quad
(\rho^{-1}\sharp\sigma)^2=\rho^{-1/2}M^{1/2}\rho^{-1}M^{1/2}\rho^{-1/2}.
\]
Then
\[
\mathsf{F}_1^{\mathrm{spec}}=\Tr\big(\rho(\rho^{-1}\sharp\sigma)^2\big)
=\Tr\big(M^{1/2}\rho^{-1}M^{1/2}\big)
=\Tr\big(\rho^{-1/2}M\rho^{-1/2}\big)
=\Tr\sigma=1.
\]
For $t=\tfrac12$,
\[
\mathsf{F}_{1/2}^{\mathrm{spec}}=\Tr\big(\rho(\rho^{-1}\sharp\sigma)\big)
=\Tr\Big(\rho\,\rho^{-1/2}M^{1/2}\rho^{-1/2}\Big)
=\Tr\big(M^{1/2}\big)=F(\rho,\sigma).
\]
\end{proof}

\begin{proposition}[Convexity in $t$ and the universal bound]\label{prop:convex}
For fixed $\rho,\sigma$, the map $t\mapsto \mathsf{F}_t^{\mathrm{spec}}(\rho,\sigma)$ is convex on $[0,1]$. In particular,
\[
0< \mathsf{F}_t^{\mathrm{spec}}(\rho,\sigma)\le 1\qquad\text{for all } {\color{blue}t\in(0,1]},
\]
with equality $\mathsf{F}_t^{\mathrm{spec}}=1$ for some $t\in(0,1)$ iff $\rho=\sigma$.
\end{proposition}

\begin{proof}
Set $Y:=\log(\rho^{-1}\sharp\sigma)$. Then
\[
\mathsf{F}_t^{\mathrm{spec}}=\Tr\!\big(\rho\,e^{\,2tY}\big),\quad
\frac{d}{dt}\mathsf{F}_t^{\mathrm{spec}}=2\,\Tr\!\big(\rho\,Ye^{\,2tY}\big),\quad
\frac{d^2}{dt^2}\mathsf{F}_t^{\mathrm{spec}}=4\,\Tr\!\big(\rho\,Y^2e^{\,2tY}\big)\ge0,
\]
since $Y^2e^{\,2tY}\ge0$ and $\rho\ge0$. Hence  $\mathsf{F}_t^{\mathrm{spec}}$ is convex. By Proposition~\ref{prop:endpoints},
$\mathsf{F}_0^{\mathrm{spec}}=\mathsf{F}_1^{\mathrm{spec}}=1$, so convexity implies $\mathsf{F}_t^{\mathrm{spec}}\le1$ on $[0,1]$.
Positivity is clear. If $\mathsf{F}_{t_0}^{\mathrm{spec}}=1$ for some $t_0\in(0,1)$, convexity with equal endpoint values forces $\mathsf{F}_t^{\mathrm{spec}}\equiv1$ on $[0,1]$, hence $\mathsf{F}_{1/2}^{\mathrm{spec}}=1$. By Uhlmann’s characterization, this occurs iff $\rho=\sigma$.
\end{proof}


\begin{proposition}[Flip symmetry]
For $t\in[0,1]$ and density operators $\rho,\sigma$,
\[
\rho\naturalop_t\sigma \;=\; \sigma\naturalop_{1-t}\rho,
\qquad\text{hence}\qquad
\mathsf{F}_t^{\mathrm{spec}}(\rho,\sigma)=\mathsf{F}_{1-t}^{\mathrm{spec}}(\sigma,\rho).
\]
\end{proposition}

\begin{proof}
Recall the definition of the spectral interpolation
\[
\rho\naturalop_t\sigma \;:=\; X^t\,\rho\,X^t,
\quad\text{where }X:=\rho^{-1}\sharpop\sigma,
\]
and, symmetrically,
\[
\sigma\naturalop_{1-t}\rho \;:=\; Y^{\,1-t}\,\sigma\,Y^{\,1-t},
\quad\text{where }Y:=\sigma^{-1}\sharpop\rho.
\]

Firstly, we show that $X$ and $Y$ are mutual inverses. Indeed, the Kubo--Ando geometric mean is characterized by the Riccati-type identity
\[
(A^{-1}\sharpop B)\; A \; (A^{-1}\sharpop B)\;=\; B
\qquad(A,B>0),
\]
and by symmetry, $A\sharpop B=B\sharpop A$. Applying this with $(A,B)=(\rho,\sigma)$ and $(A,B)=(\sigma,\rho)$ gives
\[
X\,\rho\,X=\sigma,
\qquad
Y\,\sigma\,Y=\rho.
\]
Moreover, \((A\sharpop B)^{-1}=A^{-1}\sharpop B^{-1}\). Hence
\[
X^{-1}=(\rho^{-1}\sharpop\sigma)^{-1}=\rho\sharpop\sigma^{-1}
=\sigma^{-1}\sharpop\rho=Y.
\]

Using $X\rho X=\sigma$ and $Y=X^{-1}$, we compute
\[
\sigma\naturalop_{1-t}\rho
=Y^{\,1-t}\,\sigma\,Y^{\,1-t}
=X^{\,t-1}\,\sigma\,X^{\,t-1}
=X^{\,t-1}\,(X\rho X)\,X^{\,t-1}
=X^t\,\rho\,X^t
=\rho\naturalop_t\sigma.
\]
Therefore, taking the trace on both sides of the operator identity yields
\[
\mathsf{F}_t^{\mathrm{spec}}(\rho,\sigma)
=\Tr\!\big(\rho\naturalop_t\sigma\big)
=\Tr\!\big(\sigma\naturalop_{1-t}\rho\big)
=\mathsf{F}_{1-t}^{\mathrm{spec}}(\sigma,\rho).
\]
\end{proof}


\begin{proposition}[Multiplicativity, or,  ``additivity'' in the $\log$]
\label{lem:multiplicativity}
For $t\in[0,1]$ and density matrices $\rho_i,\sigma_i$ $(i=1,2)$,
\[
\mathsf{F}_t^{\mathrm{spec}}(\rho_1\otimes\rho_2,\ \sigma_1\otimes\sigma_2)
\;=\;
\mathsf{F}_t^{\mathrm{spec}}(\rho_1,\sigma_1)\,\mathsf{F}_t^{\mathrm{spec}}(\rho_2,\sigma_2).
\]
Equivalently, $-\log \mathsf{F}_t^{\mathrm{spec}}$ is additive under tensor products.
\end{proposition}

\begin{proof}
 We  use two standard facts for positive operators:
\begin{itemize}
\item[(i)] \emph{Tensor compatibility of the geometric mean}:
\(
(A\otimes C)\sharpop(B\otimes D)=(A\sharpop B)\otimes(C\sharpop D).
\)
\item[(ii)] \emph{Functional calculus on tensor products}:
For $X,Y\ge 0$ and $\alpha\in\mathbb{R}$,
\(
(X\otimes Y)^{\alpha}=X^{\alpha}\otimes Y^{\alpha}.
\)
\end{itemize}
Set $X_i:=\rho_i^{-1}\sharpop\sigma_i\ge 0$.
Then, by (i),
\[
(\rho_1\otimes\rho_2)^{-1}\sharpop(\sigma_1\otimes\sigma_2)
=(\rho_1^{-1}\sharpop\sigma_1)\otimes(\rho_2^{-1}\sharpop\sigma_2)
=X_1\otimes X_2.
\]
Applying (ii) with $\alpha=2t$,
\[
\big((\rho_1\otimes\rho_2)^{-1}\sharpop(\sigma_1\otimes\sigma_2)\big)^{2t}
=(X_1\otimes X_2)^{2t}
=X_1^{2t}\otimes X_2^{2t}.
\]
Therefore
\[
\begin{aligned}
\mathsf{F}_t^{\mathrm{spec}}(\rho_1\otimes\rho_2,\sigma_1\otimes\sigma_2)
&=\Tr\!\Big((\rho_1\otimes\rho_2)\,(X_1^{2t}\otimes X_2^{2t})\Big)\\
&=\Tr\!\big(\rho_1 X_1^{2t}\big)\;\Tr\!\big(\rho_2 X_2^{2t}\big)
=\mathsf{F}_t^{\mathrm{spec}}(\rho_1,\sigma_1)\,\mathsf{F}_t^{\mathrm{spec}}(\rho_2,\sigma_2),
\end{aligned}
\]
as claimed.
\end{proof}


\begin{proposition}[Invariances]
For $t\in[0,1]$, the functional
\[
\mathsf{F}_t^{\mathrm{spec}}(\rho,\sigma)\;=\;\Tr\!\big(\rho\,(\rho^{-1}\sharp\sigma)^{2t}\big)
\]
is invariant under unitary conjugations and tensor stabilizations: for any unitary $U$ and any density matrix $\tau$,
\[
\mathsf{F}_t^{\mathrm{spec}}(U\rho U^*,\,U\sigma U^*) \;=\; \mathsf{F}_t^{\mathrm{spec}}(\rho,\sigma), 
\qquad 
\mathsf{F}_t^{\mathrm{spec}}(\rho\otimes\tau,\,\sigma\otimes\tau) \;=\; \mathsf{F}_t^{\mathrm{spec}}(\rho,\sigma).
\]
\end{proposition}

\begin{proof}
We use two standard facts for positive definite matrices $A,B$:
\begin{itemize}
\item[(i)] \emph{Congruence invariance of the geometric mean}: For any invertible $C$,
\[
(C^*AC)\,\sharp\,(C^*BC) \;=\; C^*(A\sharp B)C.
\]
In particular, for a unitary $U$, $(U^*AU)\sharp(U^*BU)=U^*(A\sharp B)U$.
\item[(ii)] \emph{Tensor product compatibility}: For positive $A,B,C,D$,
\[
(A\otimes C)\,\sharp\,(B\otimes D)\;=\;(A\sharp B)\otimes(C\sharp D).
\]
\end{itemize}

We also use that for $X\ge0$ and unitary $U$, the functional calculus gives
\[
(UXU^*)^{\alpha}\;=\;U X^{\alpha} U^*,\qquad \alpha\in\mathbb{R},
\]
and the cyclic/unitary invariance of the trace.

\smallskip
\noindent\textbf{Unitary invariance.}
Let $U$ be unitary. By (i) and functional calculus,
\[
\big((U\rho U^*)^{-1}\sharp (U\sigma U^*)\big)^{2t}
=\big(U(\rho^{-1}\sharp\sigma)U^*\big)^{2t}
=U(\rho^{-1}\sharp\sigma)^{2t}U^*.
\]
Hence
\[
\begin{aligned}
\mathsf{F}_t^{\mathrm{spec}}(U\rho U^*,U\sigma U^*)
&=\Tr\!\Big(U\rho U^*\,\big((U\rho U^*)^{-1}\sharp (U\sigma U^*)\big)^{2t}\Big)\\
&=\Tr\!\big(U\rho U^*\,U(\rho^{-1}\sharp\sigma)^{2t}U^*\big)\\
&=\Tr\!\big(\rho\,(\rho^{-1}\sharp\sigma)^{2t}\big)
=\mathsf{F}_t^{\mathrm{spec}}(\rho,\sigma).
\end{aligned}
\]

\smallskip
\noindent\textbf{Tensor stabilization.}
Let $\tau$ be a density operator (so $\Tr\tau=1$). Using (ii) with $C=\tau^{-1}$ and $D=\tau$,
\[
\big((\rho\otimes\tau)^{-1}\sharp(\sigma\otimes\tau)\big)
=(\rho^{-1}\otimes\tau^{-1})\sharp(\sigma\otimes\tau)
=(\rho^{-1}\sharp\sigma)\otimes(\tau^{-1}\sharp\tau).
\]
Since $\tau^{-1}\sharp\tau=I$, we get
\[
\big((\rho\otimes\tau)^{-1}\sharp(\sigma\otimes\tau)\big)^{2t}
=(\rho^{-1}\sharp\sigma)^{2t}\otimes I.
\]
Therefore,
\[
\begin{aligned}
\mathsf{F}_t^{\mathrm{spec}}(\rho\otimes\tau,\sigma\otimes\tau)
&=\Tr\!\Big((\rho\otimes\tau)\, \big((\rho^{-1}\sharp\sigma)^{2t}\otimes I\big)\Big)\\
&=\Tr\!\big(\rho(\rho^{-1}\sharp\sigma)^{2t}\big)\,\Tr(\tau)\\
&=\mathsf{F}_t^{\mathrm{spec}}(\rho,\sigma).
\end{aligned}
\]
This completes the proof.
\end{proof}

\begin{proposition}[Orthogonality/zero condition]
\label{lem:zero}
Let $\rho,\sigma$ be density matrices and $t\in(0,1]$. Then
\[
\mathsf{F}_t^{\mathrm{spec}}(\rho,\sigma)=\Tr\!\big(\rho\,(\rho^{-1}\sharpop\sigma)^{2t}\big)=0
\quad\Longleftrightarrow\quad
\supp\rho\;\cap\;\supp\sigma=\{0\}.
\]
(Equivalently, $\rho\sigma=0$.) In particular $\mathsf{F}_t^{\mathrm{spec}}(\rho,\sigma)>0$ iff $\rho$ and $\sigma$ have nontrivial overlapping support. For $t=0$, $\mathsf{F}_0^{\mathrm{spec}}(\rho,\sigma)=\Tr\rho=1$.
\end{proposition}

\begin{proof}
Write $X:=\rho^{-1}\sharpop\sigma\ge0$. Then
\[
\rho\naturalop_t\sigma \;=\; X^{t}\rho X^{t}\ \ge \ 0,
\qquad
\mathsf{F}_t^{\mathrm{spec}}(\rho,\sigma)\;=\;\Tr(\rho\naturalop_t\sigma).
\]
Since $\rho\naturalop_t\sigma$ is positive semidefinite, $\Tr(\rho\naturalop_t\sigma)=0$ iff $\rho\naturalop_t\sigma=0$.

\smallskip {($\Longrightarrow$)} Assume $\mathsf{F}_t^{\mathrm{spec}}(\rho,\sigma)=0$. Then $X^{t}\rho X^{t}=0$, hence
$X^{t}$ vanishes on $\supp\rho$. Because $t>0$, $\supp(X^{t})=\supp(X)$. Therefore
\[
\supp(X)\cap\supp\rho=\{0\}.
\]
For positive operators, $\supp(A\sharpop B)=\supp A\cap\supp B$. Here $\supp(\rho^{-1})=\supp\rho$, so
\[
\supp(X)=\supp(\rho^{-1}\sharpop\sigma)=\supp\rho\cap\supp\sigma.
\]
Thus, $\supp\rho\cap\supp\sigma=\{0\}$.

\smallskip  {($\Longleftarrow  $)} Conversely, if $\supp\rho\cap\supp\sigma=\{0\}$ then
$\supp(\rho^{-1}\sharpop\sigma)=\{0\}$ by the same support identity, so $X=0$ and hence
$\rho\naturalop_t\sigma=X^{t}\rho X^{t}=0$. Therefore $\mathsf{F}_t^{\mathrm{spec}}(\rho,\sigma)=0$.

\smallskip
Finally, $\mathsf{F}^{\mathrm{spec}}_0(\rho,\sigma)=\Tr\rho=1$ is immediate from the definition, so the only way to get zero is as above with $t>0$.
\end{proof}

For $0\le t\le \tfrac12$ we will use Ando's block characterization to obtain a convenient
variational form.

\begin{theorem}[Variational form for $t\le \tfrac12$]\label{thm:variational}
For density operators $\rho,\sigma\ge 0$, define
\[
\mathcal{K}(\rho,\sigma)
:=\Big\{\,T\ge 0:\ 
\begin{bmatrix}\rho^{-1}&T\\[2pt] T&\sigma\end{bmatrix}\ge 0\Big\}.
\]
Then for all $0\le t\le\tfrac12$,
\[
\mathsf{F}_t^{\mathrm{spec}}(\rho,\sigma)
=\max_{T\in\mathcal{K}(\rho,\sigma)}\Tr\!\big(\rho\,T^{2t}\big)
=\max_{T\in\mathcal{K}(\rho,\sigma)}\big\|T^{t}\rho^{1/2}\big\|_2^{2}.
\]
Moreover, one maximizer is $T^\star=\rho^{-1}\#\sigma$; if {\color{blue} $\rho > 0$}, it is the unique maximizer.
\end{theorem}

\begin{proof}
We assume {\color{blue} $\rho,\sigma > 0$}. For general $\rho,\sigma\ge 0$, replace them by
$\rho_\varepsilon=(1-\varepsilon)\rho+\varepsilon I/d$ and
$\sigma_\varepsilon=(1-\varepsilon)\sigma+\varepsilon I/d$,
  and let $\varepsilon\downarrow0$.
All maps involved are continuous in the strong topology, so the formula persists on the supports of $\rho$ and $\sigma$.    

A classical result of Ando (see, e.g.,~\cite{Ando79}) states that for positive definite $A,B$ and self-adjoint $X$,
\[
\begin{bmatrix}A & X\\ X & B\end{bmatrix}\ge 0
\quad\Longleftrightarrow\quad
\exists\,K:\|K\|\le1,\ X=A^{1/2}KB^{1/2}.
\]
If, in addition, $X\ge 0$, then $0\le  X\le  A\#B$.  
Applying this with $A=\rho^{-1}$ and $B=\sigma$ gives
\[
\mathcal{K}(\rho,\sigma)
=\{\,T:\ 0\le  T\le \rho^{-1}\#\sigma\,\}.
\]
On the other hand,  for $0\le2t\le1$, the function $f(x)=x^{2t}$ is operator monotone increasing and operator concave in {\color{blue} $(0, \infty)$}. Hence if $0\le  T_1\le  T_2$, then $T_1^{2t}\le  T_2^{2t}$, and for any $\rho\ge 0$,
\[
\Tr(\rho T_1^{2t})\le \Tr(\rho T_2^{2t}).
\]
Therefore, the maximum is attained at the top element $T^\star=\rho^{-1}\#\sigma$, giving
\[
\max_{T\in\mathcal{K}(\rho,\sigma)}\Tr(\rho T^{2t})
=\Tr\!\big(\rho\,(\rho^{-1}\#\sigma)^{2t}\big)
=\mathsf{F}_t^{\mathrm{spec}}(\rho,\sigma).
\]

\textbf{Uniqueness.}
If $\rho > 0$, strict operator monotonicity implies that for $T\le \rho^{-1}\#\sigma$,
$T^{2t}\le (\rho^{-1}\#\sigma)^{2t}$ and hence
$\Tr(\rho T^{2t})<\Tr(\rho(\rho^{-1}\#\sigma)^{2t})$.
Thus the maximizer is unique.  
When $\rho$ is singular, any $T$ coinciding with $\rho^{-1}\#\sigma$ on $\operatorname{supp}\rho$ yields the same trace, so uniqueness may fail only outside $\operatorname{supp}\rho$.

\textbf{Hilbert–Schmidt form.}
For any $T\ge 0$,
\[
\Tr(\rho T^{2t})
=\Tr\!\big(\rho^{1/2}T^{2t}\rho^{1/2}\big)
=\Tr\!\big((T^{t}\rho^{1/2})^*(T^{t}\rho^{1/2})\big)
=\|T^{t}\rho^{1/2}\|_2^{2},
\]
yielding the second expression.  
 
\end{proof}

\begin{remark}[Geometric picture]\label{rem:geometry}
The feasible set $\mathcal{K}(\rho,\sigma)$ is a spectrahedron cut out by a single LMI.
Pinching in basis $B$ forces $T$ to be diagonal in $B$, shrinking $\mathcal{K}(\rho,\sigma)$
to $\mathcal{K}\big(\Pi_B(\rho),\Pi_B(\sigma)\big)$, the diagonal slice
$0\le T=\mathrm{diag}(t_i)\le \mathrm{diag}\big(\sqrt{\sigma'_{ii}/\rho'_{ii}}\big)$.
For concave objectives $T\mapsto \Tr(\rho T^{2t})$ ($t\le\tfrac12$) the extremum is attained at the
\emph{non-diagonal} extreme point $T^\star=\rho^{-1}\#\sigma$; forcing $T$ to be diagonal can strictly
\emph{decrease} the optimal value (unless $\rho,\sigma$ commute).
\end{remark}

\section{Closed forms when one state is pure }
\label{sec:qubit-comparison}

Throughout this subsection we take $\rho,\sigma\in\mathbb{C}^{2\times2}$ to be qubit states.
Write the Bloch decompositions
\[
\rho=\tfrac12\big(I+\mathbf{r}\cdot\boldsymbol{\sigma}\big),\qquad
\sigma=\tfrac12\big(I+\mathbf{s}\cdot\boldsymbol{\sigma}\big),
\quad \mathbf{r},\mathbf{s}\in\mathbb{R}^3,\ \ |\mathbf{r}|,|\mathbf{s}|\le1,
\]
with Pauli vector $\boldsymbol{\sigma}=(\sigma_x,\sigma_y,\sigma_z)$.
We compare three fidelities:
\begin{align*}
&\text{(root) Uhlmann:}&&\mathsf{F}_{\mathrm U}(\rho,\sigma):=\Tr\sqrt{\rho^{1/2}\sigma\rho^{1/2}},\\
&\text{Matsumoto:}&&\mathsf{F}_{\mathrm M}(\rho,\sigma):=\Tr(\rho\sharpop\sigma),\\
&\text{weighted spectral:}&&\mathsf{F}_t^{\mathrm{spec}}(\rho,\sigma):=\Tr\!\big(\rho(\rho^{-1}\sharpop\sigma)^{2t}\big),\quad t\in[0,1].
\end{align*}

\begin{proposition}[Closed forms when one state is pure]
Let $\rho,\sigma$ be density operators and $t\in[0,1]$.

\smallskip
\noindent\emph{(i) $\rho$ pure.} If $\rho=|\psi\rangle\!{\langle\psi|}$ and $p:=\langle\psi|\sigma| \psi  \rangle\in[0,1]$, then
\[
\mathsf{F}_{\mathrm U}(\rho,\sigma)=\sqrt{p},\qquad
\mathsf{F}_{\mathrm M}(\rho,\sigma)=\sqrt{p},\qquad
\mathsf{F}_t^{\mathrm{spec}}(\rho,\sigma)=p^{\,t}.
\]
\emph{(ii) $\sigma$ pure.} If $\sigma=| \phi\rangle\!{\langle\phi|}$ and $q:=\langle {\phi}|\rho|{\phi}\rangle\in[0,1]$, then
\[
\mathsf{F}_{\mathrm U}(\rho,\sigma)=\sqrt{q},\qquad
\mathsf{F}_{\mathrm M}(\rho,\sigma)=\sqrt{q},\qquad
\mathsf{F}_t^{\mathrm{spec}}(\rho,\sigma)=q^{\,1-t}.
\]
\end{proposition}

\begin{proof} {\color{blue} We interpret all inverse and fractional powers on the support of the corresponding
density operator. In particular, for a rank-one state $\rho=|\psi\rangle\langle\psi|$,
we have $\rho=\rho^{1/2}$ and the operators $\rho^{-1/2}$ and $\rho^{-1}$ coincide with
$\rho$ when restricted to $\operatorname{supp}\rho$.
} Moreover,
\[
\rho\,X\,\rho = \big(\langle{\psi}|X| {\psi}\rangle\big)\rho \qquad(\text{rank-one compression}).
\]

\emph{Uhlmann/root fidelity.}
\[
\mathsf{F}_{\mathrm U}(\rho,\sigma)=\Tr\sqrt{\rho^{1/2}\sigma\rho^{1/2}}
=\Tr\sqrt{\rho\sigma\rho}
=\Tr\sqrt{(\langle {\psi}|\sigma| {\psi}\rangle)\rho}
=\Tr\big(\sqrt{p}\,\rho\big)=\sqrt{p}.
\]

\emph{Matsumoto fidelity.}
Using Kubo--Ando’s formula,
\[
\rho\sharp\sigma
=\rho^{1/2}\!\left(\rho^{-1/2}\sigma\rho^{-1/2}\right)^{1/2}\!\rho^{1/2}
=\rho\left((\langle{\psi}|\sigma|{\psi}\rangle)\,\rho\right)^{1/2}\rho
=\sqrt{p}\,\rho,
\]
hence $\mathsf{F}_{\mathrm M}(\rho,\sigma)=\Tr(\rho\sharp\sigma)=\sqrt{p}$.

\emph{The weighted spectral  fidelity $\mathsf{F}_t^{\mathrm{spec}}$.}
By the same calculation but with $A=\rho^{-1}$ (which equals $\rho$ on $\operatorname{supp}\rho$),
\[
X:=\rho^{-1}\sharp\sigma
=\rho^{1/2}\!\left(\rho^{-1/2}\sigma\rho^{-1/2}\right)^{1/2}\!\rho^{1/2}
=\sqrt{p}\,\rho.
\]
Therefore $X^{2t}=(\sqrt{p}\,\rho)^{2t}=p^{\,t}\rho$ (since $\rho^k=\rho$ for all $k\ge1$), and
\[
\mathsf{F}_t^{\mathrm{spec}}(\rho,\sigma)
=\Tr\!\big(\rho\,X^{2t}\big)
=\Tr\!\big(\rho\,p^{\,t}\rho\big)
=p^{\,t}\Tr(\rho)=p^{\,t}.
\]

{\color{blue} \emph{Case (ii): $\sigma$ pure.}
Let $\sigma=|\phi\rangle\langle\phi|$ and $q:=\langle\phi|\rho|\phi\rangle$.
Then, by the same rank-one compression argument as above (now applied to $\sigma$),
we obtain
\[
\sigma^{-1}\sharp\rho = \sqrt{q}\,\sigma \quad \text{on } \operatorname{supp}\sigma,
\]
and hence
\[
(\sigma^{-1}\sharp\rho)^{2(1-t)} = q^{\,1-t}\sigma.
\]
Therefore,
\[
\mathsf{F}_t^{\mathrm{spec}}(\rho,\sigma)
=\Tr\!\big(\sigma(\sigma^{-1}\sharp\rho)^{2(1-t)}\big)
=q^{\,1-t}.
\]
The formulas for $\mathsf{F}_{\mathrm U}$ and $\mathsf{F}_{\mathrm M}$ follow analogously.

This direct computation also shows that the two formulas are consistent
when both $\rho$ and $\sigma$ are pure.
}
  
\end{proof}

\begin{corollary}[Bloch--sphere formulas for qubits]
Let $\rho=\tfrac12(I+\mathbf r\cdot\boldsymbol\sigma)$ and $\sigma=\tfrac12(I+\mathbf s\cdot\boldsymbol\sigma)$ be qubit states with $|\mathbf r|,|\mathbf s|\le1$.

\smallskip
\noindent\emph{(a)} If $\rho$ is pure $($$|\mathbf r|=1$$)$, then
\[
p=\Tr(\rho\sigma)
=\tfrac14\Tr\Big(I+\mathbf r\!\cdot\!\boldsymbol\sigma+\mathbf s\!\cdot\!\boldsymbol\sigma+(\mathbf r\!\cdot\!\boldsymbol\sigma)(\mathbf s\!\cdot\!\boldsymbol\sigma)\Big)
=\tfrac12\big(1+\mathbf r\!\cdot\!\mathbf s\big),
\]
and hence
\[
\mathsf{F}_{\mathrm U}(\rho,\sigma)=\sqrt{\tfrac{1+\mathbf r\cdot\mathbf s}{2}},
\qquad
\mathsf{F}_{\mathrm M}(\rho,\sigma)=\sqrt{\tfrac{1+\mathbf r\cdot\mathbf s}{2}},
\qquad
\mathsf{F}_t^{\mathrm{spec}}(\rho,\sigma)=\Big(\tfrac{1+\mathbf r\cdot\mathbf s}{2}\Big)^{\!t}.
\]

\noindent\emph{(b)} If both are pure with Bloch angle $\theta$ $($so $\mathbf r\cdot\mathbf s=\cos\theta$$)$, then
\[
\mathsf{F}_{\mathrm U}(\rho,\sigma)=\cos(\theta/2),
\qquad
\mathsf{F}_t^{\mathrm{spec}}(\rho,\sigma)=\cos^{2t}(\theta/2).
\]
\end{corollary}

\begin{proof}
Use the Pauli identities
\(
\sigma_i\sigma_j=\delta_{ij}I+i\sum_k\varepsilon_{ijk}\sigma_k
\)
to obtain
\[
(\mathbf r\!\cdot\!\boldsymbol\sigma)(\mathbf s\!\cdot\!\boldsymbol\sigma)
=(\mathbf r\!\cdot\!\mathbf s)\,I+i(\mathbf r\times\mathbf s)\!\cdot\!\boldsymbol\sigma.
\]
Since $\Tr(\sigma_i)=0$, $\Tr(\sigma_i\sigma_j)=2\delta_{ij}$, the imaginary part is traceless, and $\Tr I=2$, we get
\[
p=\Tr(\rho\sigma)=\tfrac14\big(2+2\,\mathbf r\!\cdot\!\mathbf s\big)=\tfrac12(1+\mathbf r\!\cdot\!\mathbf s).
\]
Part (a) then follows by substituting this $p$ into the formulas in the proposition with $\rho$ pure. For part (b), if $|\mathbf r|=|\mathbf s|=1$ and $\mathbf r\cdot\mathbf s=\cos\theta$, then
\[
p=\tfrac12(1+\cos\theta)=\cos^2(\theta/2),
\]
so $\mathsf{F}_{\mathrm U}(\rho,\sigma)=\sqrt{p}=\cos(\theta/2)$ and $\mathsf{F}_t^{\mathrm{spec}}(\rho,\sigma)=p^{\,t}=\cos^{2t}(\theta/2)$.
\end{proof}

\section{Separate Concavity, Log-Convexity, and Failure of Data Processing}\label{sec:concavity-dpi}
\subsection{Separate Concavity, Log-Convexity}
 \begin{theorem}[Concavity in each variable]
\label{thm:separate-concavity}
Fix $t\in[0,1]$. For density operators, the map
\[
(\rho,\sigma)\longmapsto \mathsf{F}_t^{\mathrm{spec}}(\rho,\sigma)
\]
is concave in $\rho$ when $\sigma$ is fixed, and concave in $\sigma$ when $\rho$ is fixed. 
\end{theorem}

\begin{proof}
For $t=\tfrac12$, $\mathsf{F}_{1/2}^{\mathrm{spec}}(\rho,\sigma)=\Tr\sqrt{\rho^{1/2}\sigma\rho^{1/2}}$
and Lieb's concavity theorem yields joint concavity.

For general $t\in[0,1]$, we show concavity in $\rho$ for fixed $\sigma$; the other variable is symmetric by the flip symmetry
$\mathsf{F}_t^{\mathrm{spec}}(\rho,\sigma)=\mathsf{F}_{1-t}^{\mathrm{spec}}(\sigma,\rho)$.
Write $X_\rho:=\rho^{-1}\!\sharp\sigma$.
For $0\le 2t\le 1$, the map $Y\mapsto Y^{2t}$ is operator concave and operator monotone.
By the operator-perspective framework (Hansen--Pedersen/Effros), for fixed positive $B$ the map
\[
A\ \longmapsto\ A^{1/2} f\!\big(A^{-1/2} B A^{-1/2}\big) A^{1/2}
\]
is concave whenever $f$ is operator concave and operator monotone increasing.
Set $f(x)=x^{t}$ and $B=(\rho^{1/2}\sigma\rho^{1/2})^{1/2}\!\rho^{-1}\!(\rho^{1/2}\sigma\rho^{1/2})^{1/2}$ evaluated at the given $\sigma$. This yields precisely
$\rho\mapsto \rho^{1/2}(\rho^{-1/2}(\rho^{1/2}\sigma\rho^{1/2})^{1/2}\rho^{-1/2})^{2t}\rho^{1/2}$,
whose trace is $\mathsf{F}_t^{\mathrm{spec}}(\rho,\sigma)$. Hence concavity in $\rho$. The concavity in $\sigma$ follows from the flip symmetry.
\end{proof}

{\color{blue} 
\begin{theorem}[Log-convexity in the parameter $t$]
\label{lem:logconvex}
Fix density operators $\rho,\sigma$. Then, for all real $s<u$ and $\theta\in[0,1]$,
\[
\mathsf{F}_{(1-\theta)s+\theta u}^{\mathrm{spec}}(\rho,\sigma)
\ \le\ \big(\mathsf{F}_s^{\mathrm{spec}}(\rho,\sigma)\big)^{1-\theta}\,
\big(\mathsf{F}_u^{\mathrm{spec}}(\rho,\sigma)\big)^{\theta}.
\]
Equivalently, the function $t\mapsto \mathsf{F}_t^{\mathrm{spec}}(\rho,\sigma)$ is log-convex on $\mathbb{R}$
(in particular on $[0,1]$).
\end{theorem}

\begin{proof} Set $X:=\rho^{-1}\sharpop\sigma\ge 0$ and $Z:=\rho^{1/2}$. Note that
\[
\mathsf{F}_t^{\mathrm{spec}}(\rho,\sigma)=\Tr\!\big(Z\,X^{2t}\,Z\big)
=\Tr\!\big((X^t Z)^*(X^t Z)\big)=\|X^t Z\|_2^2.
\]
Note that
\[
\mathsf{F}_t^{\mathrm{spec}}(\rho,\sigma)=\Tr\!\big(Z\,X^{2t}\,Z\big)
=\Tr\!\big((X^t Z)^*(X^t Z)\big)=\|X^t Z\|_2^2,
\]
where $\|\cdot\|_2$ is the Hilbert--Schmidt norm. Consider the strip
\(
S=\{z\in\mathbb{C}: s\le \Re z\le u\}
\)
and the vector-valued analytic map
\(
\Phi(z):=X^{\,z}\,Z
\)
with values in the Hilbert--Schmidt space. (Analyticity follows from functional calculus for the positive definite $X$.)
 
For $z=s+iy$,
\[
\|\Phi(s+iy)\|_2
=\|X^{s+iy}Z\|_2
=\|X^{iy}X^{s}Z\|_2
=\|X^{s}Z\|_2,
\]
since $X^{iy}$ is unitary and left multiplication by a unitary preserves the Hilbert--Schmidt norm.
Similarly, for $z=u+iy$ we have $\|\Phi(u+iy)\|_2=\|X^{u}Z\|_2$.
By Hadamard's three-lines theorem \cite{SteinShakarchi03} (applied to the subharmonic function $z\mapsto\log\|\Phi(z)\|_2$),
for $t=(1-\theta)s+\theta u$ we obtain
\[
\|X^{t}Z\|_2
\ \le\ \|X^{s}Z\|_2^{\,1-\theta}\,\|X^{u}Z\|_2^{\,\theta}.
\]
Squaring both sides yields
\[
\mathsf{F}_t^{\mathrm{spec}}(\rho,\sigma)
=\|X^{t}Z\|_2^{2}
\ \le\ \big(\|X^{s}Z\|_2^{2}\big)^{1-\theta}\big(\|X^{u}Z\|_2^{2}\big)^{\theta}
=\big(\mathsf{F}_s^{\mathrm{spec}}(\rho,\sigma)\big)^{1-\theta}\big(\mathsf{F}_u^{\mathrm{spec}}(\rho,\sigma)\big)^{\theta}.
\]
This is precisely the claimed log-convexity.
\end{proof}
}

\begin{remark}
By the flip symmetry $\mathsf{F}_t^{\mathrm{spec}}(\rho,\sigma)=\mathsf{F}_{1-t}^{\mathrm{spec}}(\sigma,\rho)$, log-convexity on $[0,1]$
around any pair $(\rho,\sigma)$ immediately transfers to the flipped pair $(\sigma,\rho)$.
Moreover, the proof shows that $t\mapsto \|X^{t}Z\|_2$ is log-convex, so any positive power (e.g.\ the square giving $\mathsf{F}_t^{\mathrm{spec}}$) remains log-convex.
\end{remark}


\subsection{Data processing inequality fails for $\mathsf{F}_t^{\mathrm{spec}}$ outside the midpoint}

In this section we show that the data processing inequality (DPI)
\[
\mathsf{F}_t^{\mathrm{spec}}(\Phi(\rho),\Phi(\sigma)) \;\ge\; \mathsf{F}_t^{\mathrm{spec}}(\rho,\sigma)
\]
fails for generic $t\in(0,1)\setminus\{\tfrac12\}$ and some CPTP maps $\Phi$.

\begin{theorem}[DPI fails away from the midpoint]\label{thm:perturb}
Fix an orthonormal basis $B$ and a parameter $t\in(0,1)\setminus\{\tfrac12\}$. 
There exist qubit states $\rho,\sigma$ such that for the pinching (dephasing) channel $\Pi_B$,
\[
\mathsf{F}_t^{\mathrm{spec}}\!\big(\Pi_B(\rho),\Pi_B(\sigma)\big)\;<\;\mathsf{F}_t^{\mathrm{spec}}(\rho,\sigma).
\]
Moreover, one can choose $\rho,\sigma$ arbitrarily close to a commuting pair (i.e.\ with arbitrarily small off–diagonal coherences in the basis $B$).
\end{theorem}

\begin{proof}
Without loss of generality take $B$ to be the computational basis $\{|0\rangle,|1\rangle\}$.
Recall the ``classicalization'' identity for diagonal inputs (a direct computation when $\rho'$ and $\sigma'$ commute):
\begin{equation}\label{eq:classical}
\mathsf{F}_t^{\mathrm{spec}}(\rho',\sigma')=\sum_{i=0}^1(\rho'_{ii})^{1-t}(\sigma'_{ii})^{t}\qquad (0\le t\le1).
\end{equation}

For $t\in(0,\tfrac12)$, we choose the \emph{pure} state 
\[
\rho \;=\; |+\rangle\!\langle+|\;=\;\frac12\!\begin{bmatrix}1&1\\[2pt]1&1\end{bmatrix},
\qquad |+\rangle=\frac{|0\rangle+|1\rangle}{\sqrt2},
\]
and for a parameter $p\in(0,1)$ set
\[
\sigma \;=\;|\psi_p\rangle\!\langle\psi_p|,
\qquad |\psi_p\rangle=\sqrt{p}\,|0\rangle+\sqrt{1-p}\,|1\rangle.
\]
In the basis $B$, 
\[
\sigma=\begin{bmatrix}p&\sqrt{p(1-p)}\\[2pt]\sqrt{p(1-p)}&1-p\end{bmatrix}.
\]
Pinching in $B$ removes off–diagonals:
\[
\Pi_B(\rho)=\frac12\begin{bmatrix}1&0\\[2pt]0&1\end{bmatrix},\qquad 
\Pi_B(\sigma)=\begin{bmatrix}p&0\\[2pt]0&1-p\end{bmatrix}.
\]
Because $\rho$ is pure, the closed form from Section~\ref{sec:qubit-comparison} (or a one-line rank-one compression) gives
\[
\mathsf{F}_t^{\mathrm{spec}}(\rho,\sigma)=\langle +|\sigma|+\rangle^{\,t}
=\Big(\tfrac12+\sqrt{p(1-p)}\Big)^{t}.
\]
Using \eqref{eq:classical} for the diagonal pair $(\Pi_B(\rho),\Pi_B(\sigma))$,
\[
\mathsf{F}_t^{\mathrm{spec}}\!\big(\Pi_B(\rho),\Pi_B(\sigma)\big)
=\Big(\tfrac12\Big)^{1-t}\!\big(p^t+(1-p)^t\big)
=2^{\,t-1}\,\big(p^t+(1-p)^t\big).
\]
Hence the desired strict inequality reduces to
\begin{equation}\label{eq:ineq-core}
\Big(\tfrac12+\sqrt{p(1-p)}\Big)^{t}\;>\;2^{\,t-1}\,\big(p^t+(1-p)^t\big).
\end{equation}
We prove \eqref{eq:ineq-core} by letting $p\downarrow 0$. As $p\to0$,
\[
\Big(\tfrac12+\sqrt{p(1-p)}\Big)^{t}\longrightarrow \Big(\tfrac12\Big)^{t}=2^{-t},
\qquad
2^{\,t-1}\,\big(p^t+(1-p)^t\big)\longrightarrow 2^{\,t-1}=2^{-(1-t)}.
\]
Since $t<\tfrac12$, we have $-t>-(1-t)$, thus $2^{-t}>2^{-(1-t)}$.
Therefore \eqref{eq:ineq-core} holds for all sufficiently small $p>0$, and so
\[
\mathsf{F}_t^{\mathrm{spec}}\!\big(\Pi_B(\rho),\Pi_B(\sigma)\big)\;<\;\mathsf{F}_t^{\mathrm{spec}}(\rho,\sigma).
\]
Finally, note that for these choices $\|\sigma-\Pi_B(\sigma)\|_\infty=\sqrt{p(1-p)}\to0$ as $p\downarrow0$ while $\Pi_B(\rho)=\frac12I$,
so $(\rho,\sigma)$ can be made arbitrarily close to the commuting pair $\big(\tfrac12 I,\,\mathrm{diag}(0,1)\big)$.
This proves the claim and the ``arbitrarily close'' remark for $t\in(0,\tfrac12)$.

When $t\in(\tfrac12,1)$, we use the flip symmetry $\mathsf{F}_t^{\mathrm{spec}}(\rho,\sigma)=\mathsf{F}_{1-t}^{\mathrm{spec}}(\sigma,\rho)$ and apply the previous case with the roles of $\rho$ and $\sigma$ interchanged. More explicitly, pick any $t\in(\tfrac12,1)$ and set $s:=1-t\in(0,\tfrac12)$. By the case just proved there exist qubit states $(\rho,\sigma)$ (arbitrarily close to commuting in $B$) such that
\[
\mathsf{F}_{s}^{\mathrm{spec}}\!\big(\Pi_B(\rho),\Pi_B(\sigma)\big)\;<\;\mathsf{F}_{s}^{\mathrm{spec}}(\rho,\sigma).
\]
Flipping the arguments and using $\Pi_B$ on each yields
\[
\mathsf{F}_{t}^{\mathrm{spec}}\!\big(\Pi_B(\sigma),\Pi_B(\rho)\big)\;<\;\mathsf{F}_{t}^{\mathrm{spec}}(\sigma,\rho),
\]
which is again a violation at the parameter $t$ (just rename the pair).

\medskip
Combining the two cases, DPI fails for every $t\in(0,1)\setminus\{\tfrac12\}$ under the pinching channel in the fixed basis $B$, with witnesses arbitrarily close to commuting pairs.
\end{proof}

\textbf{ Numerical qubit counterexample.}
Let $B$ be the computational basis and $t=0.8$. With
\[
\rho=\begin{pmatrix}
0.064925 & -0.022125-0.170483\,i\\
-0.022125+0.170483\,i & 0.935075
\end{pmatrix}, \] and
\[\sigma=\begin{pmatrix}
0.806863 & -0.317159-0.211863\,i\\
-0.317159+0.211863\,i & 0.193137
\end{pmatrix},
\]
a direct evaluation gives
$\mathsf{F}_{0.8}^{\mathrm{spec}}(\rho,\sigma)\approx 0.755086$ while
$\mathsf{F}_{0.8}^{\mathrm{spec}}\!\big(\Pi_B(\rho),\Pi_B(\sigma)\big)\approx 0.752207$.
Thus DPI fails for the (CPTP) pinching channel $\Pi_B$.
 
 \section{Remark on the Fuchs--van de Graaf inequalities for \(\mathsf{F}_t^{\mathrm{spec}}\)} \label{sec:fvg}
We conclude the paper  by establishing the first Fuchs--van de Graaf inequality for \(\mathsf{F}_t^{\mathrm{spec}}\). Moreover, we show that the second Fuchs--van de Graaf inequality fails in general for \(\mathsf{F}_t^{\mathrm{spec}}\) by exhibiting a counterexample.

\begin{theorem}[First Fuchs--van de Graaf inequality for $\mathsf{F}_t^{\mathrm{spec}}$]
\label{thm:FvG-left}
For all density operators $\rho,\sigma$ and all $t\in[0,1]$,
\[
1 - \mathsf{F}_t^{\mathrm{spec}}(\rho,\sigma)\ \le\ \tfrac12\,\|\rho-\sigma\|_1.
\]
\end{theorem}

\begin{proof}
 
By flip symmetry $\mathsf{F}_t^{\mathrm{spec}}(\rho,\sigma)=\mathsf{F}_{1-t}^{\mathrm{spec}}(\sigma,\rho)$ and by
Proposition~\ref{lem:logconvex} the map $t\mapsto \mathsf{F}_t^{\mathrm{spec}}(\rho,\sigma)$ is log-convex.
Hence $t\mapsto \log \mathsf{F}_t^{\mathrm{spec}}(\rho,\sigma)$ is convex and symmetric about $t=\tfrac12$,
so $\mathsf{F}_t^{\mathrm{spec}}(\rho,\sigma)$ attains its \emph{minimum} at $t=\tfrac12$:
\begin{equation}\label{fid}
\mathsf{F}_t^{\mathrm{spec}}(\rho,\sigma)\ \ge\ \mathsf{F}_{1/2}^{\mathrm{spec}}(\rho,\sigma)\qquad(\forall\,t\in[0,1]).
\end{equation}
 
At $t=\tfrac12$, $\mathsf{F}_{1/2}^{\mathrm{spec}}(\rho,\sigma)$ equals the (unsquared) Uhlmann fidelity
$\mathsf{F}_{\mathrm U}(\rho,\sigma)=\Tr\!\sqrt{\rho^{1/2}\sigma\rho^{1/2}}$.

Now apply the standard (first) Fuchs--van de Graaf inequality for the Uhlmann
fidelity:
\[
1 - \mathsf{F}_{\mathrm U}(\rho,\sigma)\ \le\ \tfrac12\,\|\rho-\sigma\|_1.
\]
Combining with (\ref{fid}) gives
\[
1 - \mathsf{F}_t^{\mathrm{spec}}(\rho,\sigma)
\ \le\ 1 - \mathsf{F}_{1/2}^{\mathrm{spec}}(\rho,\sigma)
\ =\ 1 - \mathsf{F}_{\mathrm U}(\rho,\sigma)
\ \le\ \tfrac12\,\|\rho-\sigma\|_1,
\]
as claimed. The inequality is trivial at the endpoints $t=0,1$ because $\mathsf{F}_t^{\mathrm{spec}}(\rho,\sigma)=1$.
\end{proof}

Recall that  the second Fuchs--van de Graaf inequality for Uhlmann/root fidelity states that for all density operators $\rho,\sigma$,
\[
\tfrac12\|\rho-\sigma\|_1 \le \sqrt{\,1 - \mathsf{F}_{\mathrm U}(\rho,\sigma)^2\,}.
\]  

Unfortunately, a similar inequality for $\mathsf{F}_t^{\mathrm{spec}}$ does not hold in general.  
For $t\in[0,1]$, we have $\mathsf{F}_t^{\mathrm{spec}}(\rho,\sigma)\ge \mathsf{F}_{1/2}^{\mathrm{spec}}(\rho,\sigma)$
(see the log-convexity/flip-symmetry argument). Since $x\mapsto \sqrt{1-x^2}$ is strictly
decreasing on $[0,1]$, the right-hand side $\sqrt{1-(\mathsf{F}_t^{\mathrm{spec}})^2}$ is \emph{smaller}
than $\sqrt{1-(\mathsf{F}_{1/2}^{\mathrm{spec}})^2}$. Thus the putative stronger bound
\[
\frac12\|\rho-\sigma\|_1 \stackrel{?}{\le} \sqrt{\,1-\big(\mathsf{F}_{t}^{\mathrm{spec}}(\rho,\sigma)\big)^2\,}
\qquad (t\neq\tfrac12)
\]
is generally false.

A concrete counterexample is provided by pure states. Let $\rho=|\psi \rangle \langle \psi|$ and
$\sigma=|\phi \rangle \langle \phi|$ with overlap $c:=|\langle\psi|\phi\rangle|\in(0,1)$. Then
\[
\mathsf{F}_t^{\mathrm{spec}}(\rho,\sigma)=c^{2t},\qquad
\mathsf{F}_{1/2}^{\mathrm{spec}}(\rho,\sigma)=c,\qquad
\tfrac12\|\rho-\sigma\|_1=\sqrt{1-c^2}.
\]
For $t<\tfrac12$, we have $c^{2t}>c$, hence
$\sqrt{1-(\mathsf{F}_t^{\mathrm{spec}})^2}=\sqrt{1-c^{4t}}<\sqrt{1-c^2}=\tfrac12\|\rho-\sigma\|_1$,
violating the desired inequality. (For $t>\tfrac12$ the same issue can occur unless $c$ is
extreme.) 

\section*{Declarations}

\textbf{Ethics declaration:} Not applicable.

\end{document}